\documentclass[11pt,a4paper]{article}

\usepackage{authblk}
\usepackage[margin=1in]{geometry}
\usepackage{t1enc}
\usepackage[utf8]{inputenc}
\usepackage{amsthm,amsmath,amssymb}
\usepackage{graphicx}
\usepackage{enumerate}
\usepackage{hyperref}
\usepackage{bm}
\usepackage{comment}
\usepackage{amsfonts}
\usepackage{graphicx,caption}
\usepackage{bm}
\usepackage{amsmath, amsthm, amssymb}
\usepackage{graphicx}
\usepackage{hyperref}
 \usepackage{relsize}
\usepackage{algpseudocode}

\usepackage{bbm}

\theoremstyle{plain}
\usepackage{amsthm}
\makeatletter
\newcommand{\newreptheorem}[2]{\newtheorem*{rep@#1}{\rep@title}\newenvironment{rep#1}[1]{\def\rep@title{#2 \ref*{##1}}\begin{rep@#1}}{\end{rep@#1}}}
\makeatother

\newtheorem{theorem}{Theorem}[section]
\newtheorem*{theorem-non}{Theorem}
\newtheorem*{non-lemma}{Lemma}
\newtheorem{lemma}[theorem]{Lemma}
\newreptheorem{lemma}{Lemma}

\theoremstyle{definition}
\newtheorem{remark}[theorem]{Remark}

\DeclareMathOperator{\Aut}{Aut}

\DeclareMathOperator{\dkl}{D_{KL}}

\begin{document}
\title{Bounds on the mod 2
homology of random 2-dimensional determinantal hypertrees}
\author{Andr\'as M\'esz\'aros\thanks{HUN-REN Alfr\'ed R\'enyi Institute of Mathematics, \\Budapest, Hungary,\\ {\tt meszaros@renyi.hu}
}}
\date{}

\maketitle
\begin{abstract}
As a first step towards a conjecture of Kahle and Newman, we prove that if $T_n$ is a random $2$-dimensional determinantal hypertree on $n$ vertices, then 
    \[\frac{\dim H_1(T_n,\mathbb{F}_2)}{n^2}\]
    converges to zero in probability.

    Confirming a conjecture of Linial and Peled, we also prove the analogous statement for the $1$-out $2$-complex.

    Our proof relies on the large deviation principle for the Erd\H{o}s-R\'enyi random graph by Chatterjee and Varadhan.


\end{abstract}

\section{Introduction}

\emph{Determinantal hypertrees} are natural higher dimensional generalizations of a uniform random spanning tree of a  complete graph. They can be defined in any dimension, but in this paper, we restrict our attention to the $2$-dimensional case. A $2$-dimensional simplicial complex $S$ on the vertex set $[n]=\{1,2,\dots,n\}$ is called a ($2$-dimensional) hypertree, if
\begin{enumerate}[\hspace{30pt}(a)]
    \item\label{pra} $S$ has complete $1$-skeleton;
    \item\label{prb} The number of triangular faces of $S$ is ${n-1}\choose{2}$;
    \item\label{prc} The homology group $H_{1}(S,\mathbb{Z})$ is finite.
\end{enumerate}

In one dimension, a spanning tree must be connected, property \eqref{prc} above is the two dimensional analogue of this requirement. Note that any complex $S$ satisfying \eqref{pra} and \eqref{prc} must have at least ${n-1}\choose{2}$ triangular faces. For a graph $G$, if the reduced homology group $\tilde{H}_0(G,\mathbb{Z})$ is finite, then it is trivial. This statement fails in two dimensions since for a hypertree $S$, the order of $H_1(S,\mathbb{Z})$ can range from $1$ to $\exp(\Theta(n^2))$, see \cite{kalai1983enumeration}. Thus, while the homology of spanning trees is uninteresting, the homology of $2$-dimensional hypertrees is a very rich subject to study.

Kalai's generalization of Cayley's formula \cite{kalai1983enumeration} states that
\[\sum |H_{1}(S,\mathbb{Z})|^2=n^{{n-2}\choose {2}},\]
where the summation is over all the  hypertrees $S$ on the vertex set $[n]$. This formula suggests that the natural probability measure on the set of  hypertrees is the one where the probability assigned to a hypertree $S$ is \begin{equation}\label{measuredef}
    \frac{|H_{1}(S,\mathbb{Z})|^2}{n^{{n-2}\choose {2}}}.
\end{equation}
It turns out that this measure is a determinantal probability measure \cite{lyons2003determinantal,hough2006determinantal}. Thus, a random hypertree $T_n$ distributed according to \eqref{measuredef} is called a determinantal hypertree. General random determinantal complexes were investigated by Lyons \cite{lyons2009random}. While uniform random spanning trees are well-studied \cite{ald1,ald2,ald3,grimmett1980random,szekeres2006distribution,lyons2017probability}, a theory of determinantal hypertrees started to emerge only recently  \cite{kahle2022topology,meszaros2022local,werf2022determinantal,linial2019enumeration,meszaros2023coboundary,meszaros20242}.

It is quite natural to investigate the mod $2$ homology of random complexes. For example, one possible higher dimensional generalization of  expander graphs are the so-called coboundary expanders \cite{lubotzky2018high,gromov2010singularities,linial2006homological}, which are defined using the expansion properties of cochains with $\mathbb{F}_2$ coefficients.  Applications of coboundary
expanders include coding theory and property testing \cite{codes1,codes2,codes3,kaufman2014high}. 

Our main theorem is the following.
\begin{theorem}\label{thm1}
    Let $T_n$ be a $2$-dimensional determinantal hypertree on $n$ vertices. Then
    \[\frac{\dim H_1(T_n,\mathbb{F}_2)}{n^2}\]
    converges to zero in probability.
\end{theorem}

We give an outline of the proof. Over $\mathbb{F}_2$ cochains are uniquely determined by their support. Since $T_n$ has complete $1$-skeleton, $C^1(T_n,\mathbb{F}_2)$ can be identified with the set of graphs on the vertex set $[n]$. Thus,
\begin{equation}\label{introsum}\mathbb{E}|Z^1(T_n,\mathbb{F}_2)|=\sum_G \mathbb{P}(G\in Z^1(T_n,\mathbb{F}_2)),\end{equation}
where the summation is over all graphs $G$ on $[n]$. Once we prove that this sum is $\exp(o(n^2))$, Theorem~\ref{thm1} follows using Markov's inequality. To estimate the sum~in~\eqref{introsum}, for $u>0$, we need an upper bound on the number of graphs $G$ such that $\mathbb{P}(G\in Z^1(T_n,\mathbb{F}_2))\ge \exp\left(-\frac{n^2}2 u\right)$. We obtain such an estimate by applying the \emph{large deviation principle} for the Erd\H{o}s-R\'enyi random graph by Chatterjee and Varadhan \cite{chatterjee2011large}. This large deviation principle is formulated using the language of \emph{dense graph limit} theory, which was developed by  Lov\'asz, Szegedy, T.~S\'os,  Borgs, Chayes, Vesztergombi,  Schrijver, Freedman and many others \cite{dl1,dl2,dl3,dl4,dl5,dl6,dl7,dl8,dl9,dl10}.  See the book of Lov\'asz \cite{lovasz2012large} for an overview of the area of graph limit theory. To apply the large deviation principle, we prove an estimate of the form \[\log \mathbb{P}(G\in Z^1(T_n,\mathbb{F}_2))\le (1+o(1))\frac{n^2}2 f(G)+o(n^2),\]
where $f$ is not only defined on graphs, but it also extends to an upper semicontinuous function on the space of graphons, which is a completion of the space of graphs. Upper semicontinuity is crucial here, because $f^{-1}([-u,\infty))$ needs to be closed in order to be able to apply the large deviation principle. Finally, we show that the estimates provided by the large deviation principle are strong enough to conclude that indeed  $\mathbb{E}|Z^1(T_n,\mathbb{F}_2)|=\exp(o(n^2))$.

The \emph{$1$-out $2$-complex} $S_2(n,1)$ is defined as follows. We start with the complete graph on the vertex set $[n]$, and independently for each edge $\{u,v\}\in {{[n]}\choose{2}}$, we choose a third vertex $w$ from the set $[n]\setminus \{u,v\}$ uniformly at random and add the triangular face $\{u,v,w\}$ to the complex. If a triangular face is chosen at multiple edges, we only keep one copy of that face. 

Our next theorem confirms a conjecture of Linial and Peled \cite[Section 5]{linial2019enumeration}.

\begin{theorem}\label{thm2}
Consider the $1$-out $2$-complex $S_2(n,1)$, then   
    \[\frac{\dim H_1(S_2(n,1),\mathbb{F}_2)}{n^2}\]
    converges to zero in probability.
\end{theorem}

Linial and Peled \cite{linial2019enumeration} proved a version of Theorem~\ref{thm2} where the field $\mathbb{F}_2$ is replaced with~$\mathbb{R}$. In another paper,  Linial and Peled \cite{linial2016phase} also  proved similar results on the real homology of the \emph{Linial-Meshulam complex} $Y_2(n,p)$. This random complex is a natural higher dimensional analogue of  Erd\H{o}s-R\'enyi random graphs and it was introduced by Linial and Meshulam \cite{linial2006homological} as follows: Let $Y_2(n,p)$ be a random complex on $n$ vertices with a complete $1$-skeleton where each triangular face is added with probability $p$ independently. Linial and Peled \cite{linial2016phase} proved that for any $c\ge 0$,
\[\frac{\dim H_1\left(Y_2\left(n,\frac{c}n\right),\mathbb{R}\right)}{n^2}\]
converges in probability to a constant $\beta_1(c)$, where $\beta_1(c)$ is determined by a certain fixed point equation. They also proved that there is a critical $c^*$, which separates the linear behaviour of $\beta_1(c)$ from the non-linear one. (In \cite{linial2019enumeration} and \cite{linial2016phase}, the results are formulated in terms of $H_2$, but using Euler's formula it is easy to rewrite these results in terms of $H_1$. Also, both papers cover the case $d>2$ too.)  Linial and  Peled build on the methods of Bordenave, Lelarge and Salez~\cite{bordenave2011rank}. In contrast to the present paper,  the paper \cite{bordenave2011rank} uses notions from sparse graph limit theory~\cite{aldous2016processes,lovasz2012large}. It also relies on spectral methods which are not available in positive characteristics as it is the case in Theorems~\ref{thm1} and \ref{thm2}. Note that  with real coefficients the homology of determinantal hypertrees is uninteresting, because it is straightforward from the definitions that $\dim H_1(T_n,\mathbb{R})=0$.  

See \cite{rc1,rc2,rc3,rc4,rc5,rc6,rc7,rc8,rc9,rc10} for further results on the homology of random simplicial complexes. 

Kahle and Newman~\cite{kahle2022topology} conjectured that $\dim H_1(T_n,\mathbb{F}_2)$ is of constant order, that is, the sequence of random variables $\dim H_1(T_n,\mathbb{F}_2)$ is tight. They had an even stronger conjecture, namely, they conjectured that for any prime $p$, the $p$-torsion $\Gamma_{n,p}$ of $H_1(T_n,\mathbb{Z})$  converges to the \emph{Cohen-Lenstra distribution}. By now we know that this stronger conjecture is not true for $p=2$~\cite{meszaros20242}, but we still believe that $\dim H_1(T_n,\mathbb{F}_2)$ is of constant order. For $p>2$, it is still open whether $\Gamma_{n,p}$ of $H_1(T_n,\mathbb{Z})$  converges to the Cohen-Lenstra distribution, that is, whether  for any finite abelian $p$-group $G$, we have
\[\lim_{n\to\infty} \mathbb{P}(\Gamma_{n,p}\cong G)=\frac{1}{|\Aut(G)|}\prod_{j=1}^{\infty}\left(1-p^{-j}\right).\]
This conjecture would imply that
\[\lim_{n\to\infty} \mathbb{P}(\dim  H_1(T_n,\mathbb{F}_p)=r)=p^{-r^2} \prod_{j=1}^{r} \left(1-p^{-j}\right)^{-2} \prod_{j=1}^{\infty}\left(1-p^{-j}\right).\]
The same should be also true if we consider the uniform distribution on the hypertrees, see~\cite{kahle2020cohen} for some numerical evidence. 
See the survey of Wood~\cite{wood2022probability} for more information on the Cohen-Lenstra heuristics. For a related model motivated by determinantal hypertrees, the Cohen-Lenstra limiting distribution was established by the author for $p\ge 5$ \cite{meszaros2023cohen}.

In light of the conjecture above it would be interesting to replace the denominator $n^2$ in Theorem~\ref{thm1} with something smaller.

Determinantal hypertrees and $1$-out complexes seem to be very similar in several aspects, for example they also have the same local weak limit \cite{meszaros2022local}. Can we give a general explanation for this phenomenon? Maybe there is a coupling of $T_n$ and $S_2(n,1)$ such that their symmetric difference is small under this coupling. 

\bigskip

\textbf{Acknowledgement:} The author is grateful to Elliot Paquette, Mikl\'os Ab\'ert and B\'alint Vir\'ag for the useful discussions.
The author was supported by the NSERC discovery grant of B\'alint Vir\'ag, the KKP 139502 project, the Dynasnet European Research Council Synergy project -- grant number ERC-2018-SYG 810115, and the NKKP-STARTING 150955 project.

\section{Preliminaries on dense graph limits}

\subsection{The space of graphons}

The notations are not uniform in the literature, we mainly follow the notations of the book~\cite{lovasz2012large}.

Let $\mathcal{W}$ be the set of all bounded symmetric measurable functions $W:[0,1]^2\to\mathbb{R}$. The elements of $\mathcal{W}$ will be called kernels. Let $\mathcal{W}_0=\{W\in\mathcal{W}\,:\, 0\le W\le 1\}$, the elements of this set will be called graphons.

Given a graph $G$ on the vertex set $\{1,2,\dots,n\}$, the corresponding graphon $W_G$ is defined by
\[W_G(x,y)=\begin{cases}
1&\text{if $(\lceil xn \rceil,\lceil yn \rceil)$ is an edge of $G$},\\
0&\text{otherwise.}
\end{cases}\]

The constant $1$ graphon will be denoted by $\mathbbm{1}$.

Given a kernel $W\in \mathcal{W}$, we define the operator $T_W:L_{\infty}([0,1])\to L_{\infty}([0,1])$ by the formula
\[(T_Wf)(x)=\int_0^1 W(x,y)f(y)dy.\]
The same formula can be used to define an operator from $L_{\infty}([0,1])$ to $ L_{1}([0,1])$. With a slight abuse of notation we will also use $T_W$ to denote this operator.

Given two kernels $V,W\in \mathcal{W}$, their operator product $V\circ W$, which is a bounded measurable function from $[0,1]^2$ to $\mathbb{R}$, is defined as
\[(V\circ W)(x,y)=\int_0^1 V(x,t)W(t,y)dt.\]
Note that $V\circ W$ is not necessarily symmetric, but it will not be a problem for us. Also, in this paper $V\circ W$ will be symmetric in most cases anyway. The motivation behind this definition is the identity $T_VT_W=T_{V\circ W}$. (Note that the formula for $T_W$ makes sense even if we drop the requirement that $W$ is symmetric.)

We define the $L_\infty$ and $L_1$ norms on $\mathcal{W}$ by
\[\|W\|_\infty=\sup_{(x,y)\in [0,1]^2} |W(x,y)|\qquad\text{ and }\qquad\|W\|_1=\int_{[0,1]^2} |W(x,y)| dxdy.\]

We define the cut norm on $\mathcal{W}$ by
\[\|W\|_\square=\sup_{S,T} \left|\int_{S\times T} W(x,y) dxdy\right|,\]
where the supremum is over all measurable subsets $S,T$ of $[0,1]$. 

Clearly $\|W\|_\square\le \|W_1\|$. 

Finally, we can also consider the norm of $T_W$ as an operator $L_{\infty}([0,1])\to L_1([0,1])$. This norm turns out to be equivalent to the cut norm as we have the inequalities
\begin{equation}\label{squarevsoperator}\|W\|_\square\le \|T_W\|_{\infty\to 1} \le 4\|W\|_\square,
\end{equation}
see \cite[Lemma 8.11]{lovasz2012large}. Note that the norms above also make sense if we drop the requirement that $W$ is symmetric.

Let $S_{[0,1]}$ be the set of invertible measure preserving maps $[0,1]\to [0,1]$. Given $\varphi\in S_{[0,1]}$ and $W\in \mathcal{W}$, we define $W^\varphi\in \mathcal{W}$ by
\[W^{\varphi}(x,y)=W(\varphi(x),\varphi(y)).\]

The cut distance of two kernels $V,W\in \mathcal{W}$ is defined as
\[\delta_\square(V,W)=\inf_{\varphi\in S_{[0,1]}} \|V-W^\varphi\|_\square.\]

The cut distance only gives a pseudometric, since different kernels can have distance zero. Identifying the kernels at distance zero, we obtain the set $\widetilde{\mathcal{W}}$ of unlabeled kernels. The equivalence class of a kernel $W$ will be denoted by $\widetilde{W}$. The set of unlabeled graphons $\widetilde{\mathcal{W}}_0$ is defined analogously. By a closed subset of $\widetilde{\mathcal{W}}_0$, we mean a closed subset of the metric space $(\widetilde{\mathcal{W}}_0,\delta_\square)$.

\subsection{A few continuity results}

\begin{lemma}\label{normofproduct}
Let $V,W\in\mathcal{W}$, then
\[\|V\circ W\|_1\le 4\|V\|_
\square \|W\|_{\infty}\quad\text{ and }\quad\|W\circ V\|_1\le 4\|V\|_
\square \|W\|_{\infty}.\]
\end{lemma}
\begin{proof}
For $y\in[0,1]$, let  $f_{y}(x)=W(x,y)$ and $g_y(x)=(V\circ W)(x,y)$. Note that $g_y=T_V f_y$. Then
\begin{align*}\|V\circ W\|_1=\int_0^1 \|g_y\|_1 dy&=\int_0^1 \|T_V f_y\|_1dy \\&\le \int_0^1 \|T_V\|_{\infty \to 1 } \|f_y\|_{\infty} dy\le \|T_V\|_{\infty \to 1 }\|W\|_{\infty}\le 4\|V\|_
\square \|W\|_{\infty},\end{align*}
where the last inequality follows from \eqref{squarevsoperator}. Thus, the first inequality follows.  Since $V$ and $W$ are symmetric, we have $(V\circ W)(x,y)=(W\circ V)(y,x)$. Therefore, $\|W\circ V\|_1=\|V\circ W\|_1$. Thus, the second inequality follows from the first one.
\end{proof}

\begin{lemma}\label{limitofproduct}
Let $V,V_1,V_2,\dots$ and $W,W_1,W_2\dots$ be kernels. Assume that there is a constant~$C$ such that  $\|V_n\|_\infty\le C$ for all $n$. If $\lim_{n\to\infty} V_n=V$ and $\lim_{n\to\infty} W_n=W$ in the cut norm, then $\lim_{n\to\infty} V_n\circ W_n=V\circ W$ in the $L_1$ norm.  
\end{lemma}
\begin{proof}
Using Lemma \ref{normofproduct}, we have
\begin{align*}\|V_n\circ W_n-V\circ W\|_1&\le \|V_n\circ(W_n-W)\|_1+\|(V_n-V)\circ W\|_1
\\&\le 4\|W_n-W\|_\square \|V_n\|_\infty+4\|V_n-V\|_\square \|W\|_\infty\\&\le 4C\|W_n-W\|_\square +4\|V_n-V\|_\square \|W\|_\infty.
\end{align*}
Thus, the statement follows.
\end{proof}

Given two kernels $V,W\in\mathcal{W}$, we define
\[\langle V,W\rangle=\int_{[0,1]^2} V(x,y)W(x,y) dxdy.\]

Note that in certain cases this definition makes sense even if we drop the requirement that $V$ and $W$ are bounded. For example, $\langle V,W\rangle$ is defined if $V(x,y)W(x,y)\le 0$ for all $(x,y)\in [0,1]^2$, of course, it might be equal to $-\infty$.

\begin{lemma}\label{lemmaangle0}
    Let $V_1,V_2,\dots$ and $W$ be kernels. Assume that there is a constant $C$ such that   $\|V_n\|_\infty\le C$ for all $n$. If $\lim_{n\to\infty} \|V_n\|_\square=0$, then $\lim_{n\to\infty} \langle V_n,W\rangle =0$.
\end{lemma}
\begin{proof}
This follows from \cite[Lemma 8.22]{lovasz2012large}. (In the statement \cite[Lemma 8.22]{lovasz2012large} it is assumed that $C=1$, but this is clearly true for any $C$.)
\end{proof}

\begin{lemma}\label{lemmaangle}
    Let $V,V_1,V_2,\dots$ and $W,W_1,W_2\dots$ be kernels. Assume that there is a constant~$C$ such that $\|V\|_\infty\le C$ and   $\|V_n\|_\infty\le C$ for all $n$. If $\lim_{n\to\infty} V_n=V$ in the cut norm and $\lim_{n\to\infty} W_n=W$ in the $L_1$ norm, then $\lim_{n\to\infty} \langle V_n,W_n\rangle =\langle V,W\rangle$.
\end{lemma}
\begin{proof}
We have
\[\langle V_n,W_n\rangle-\langle V,W\rangle=\langle V_n,W_n-W\rangle+\langle V_n-V,W\rangle.\]
Since $\|V_n-V\|_\infty\le 2C$, we have $\lim_{n\to\infty} \langle V_n-V,W\rangle=0$ by Lemma~\ref{lemmaangle0}. Also
\[|\langle V_n,W_n-W\rangle|\le \|V_n\|_\infty \|W_n-W\|_1\le C\|W_n-W\|_1.\]
Thus, $\lim_{n\to\infty} \langle V_n,W_n-W\rangle=0$.
\end{proof}
\begin{remark}
    In Lemma \ref{lemmaangle}, the condition that $\lim_{n\to\infty} W_n=W$ in the $L_1$ norm can not be relaxed to convergence in cut norm as the following example shows. By considering the graphons corresponding to Erd\H{o}s-R\'enyi random graph, one can find a sequence of $0-1$ valued graphons $W_n$ converging to $\frac{1}2\mathbbm{1}$. Then $\lim_{n\to\infty}\langle W_n,W_n\rangle=\lim_{n\to\infty}\langle W_n,\mathbbm{1}\rangle=\frac{1}2$, but $\langle \frac{1}2\mathbbm{1},\frac{1}2\mathbbm{1}\rangle=\frac{1}4$.
\end{remark}

Given a measurable function $f:\mathbb{R}\to\mathbb{R}$ and a kernel $W\in \mathcal{W}$, $f\circ W$ is a symmetric measurable function from $[0,1]^2$ to $\mathbb{R}$. If $f$ is bounded, then $f\circ W\in \mathcal{W}$.

The proof of the next lemma is straightforward.

\begin{lemma}\label{Lipschitz}
Assume that $\lim_{n\to\infty} W_n=W$ in $L_1$-norm. If $f:\mathbb{R}\to\mathbb{R}$ is a bounded Lipschitz function, then $\lim_{n\to\infty} f\circ W_n=f\circ W$ in $L_1$-norm. 
\end{lemma}

\subsection{The large deviation principle for the Erd\H{o}s-R\'enyi random graph}

For $p\in (0,1)$, let $I_{p}:[0,1]\to \mathbb{R}$ be the function
\[I_{p}(u)=\frac{1}2 u\log\frac{u}{p}+\frac{1}2 (1-u)\log\frac{1-u}{1-p}.\]

For a graphon $W\in\mathcal{W}_0$, we define
\[I_p(W)=\int_{[0,1]^2} I_p(W(x,y))dxdy.\]

The function $I_p$ is also well-defined on $\widetilde{\mathcal{W}}_0$, see \cite[Lemma 2.1]{chatterjee2011large}. (To avoid confusion, we mention that the authors of \cite{chatterjee2011large} use $\mathcal{W}$ and $\widetilde{\mathcal{W}}$ to denote the space of graphons, in this paper, we follow the convention of \cite{lovasz2012large} and use $\mathcal{W}_0$ and $\widetilde{\mathcal{W}}_0$.)

Let $G(n,p)$ be the Erd\H{o}s-R\'enyi random graph, that is, the random graph on $n$
vertices where each edge is added independently with probability $p$.

\begin{theorem}[Chatterjee and Varadhan \cite{chatterjee2011large}]\label{largediv0} Let $p\in (0,1)$, $\widetilde{F}$ be a closed subset of $\widetilde{\mathcal{W}}_0$, and let $W_n=W_{G(n,p)}$. Then
\[\limsup_{n\to\infty} \frac{1}{n^2}\log \mathbb{P}(\widetilde{W}_n\in \widetilde{F})\le -\inf_{\widetilde{V}\in \widetilde{F}} I_p(\widetilde{V}).\]
\end{theorem}

Let $H:[0,1]\to \mathbb{R}$ be the function
\[H(u)=-u\log(u)- (1-u)\log(1-u).\]

For a graphon $W\in\mathcal{W}_0$, we define
\[H(W)=\int_{[0,1]^2} H(W(x,y))dxdy.\]

Note that \begin{equation}\label{HvsI}H(W)=\log(2)-2I_{\frac{1}2}(W).\end{equation} Since $I_{\frac{1}2}$ is well defined on $\widetilde{\mathcal{W}}_0$, so is $H$.

For a subset $\widetilde{F}$ of $\widetilde{\mathcal{W}}_0$, let
\[\mathcal{G}(n,\widetilde{F})=\{G\,:\,G\text{ is a graph on the vertex set }\{1,2,\dots,n\}\text{ such that } \widetilde{W}_G\in \widetilde{F}\}.\]

Letting $\widetilde{W}_n=\widetilde{W}_{G(n,1/2)}$, we see that
\[|\mathcal{G}(n,\widetilde{F})|=2^{{n}\choose{2}}\mathbb{P}(\widetilde{W}_n\in \widetilde{F})=2^{(1+o(1))\frac{n^2}2}\mathbb{P}(\widetilde{W}_n\in \widetilde{F}).\]

Thus, as straightforward corollary of Theorem~\ref{largediv0} and \eqref{HvsI}, we obtain the following statement.
\begin{lemma}\label{logcount}
Let $\widetilde{F}$ be a closed subset of $\widetilde{\mathcal{W}}_0$. Then
\[\limsup_{n\to\infty} \frac{1}{n^2}\log |\mathcal{G}(n,\widetilde{F})|\le \frac{1}2\sup_{\widetilde{V}\in \widetilde{F}} H(\widetilde{V}).\]

\end{lemma}

\section{Bounding the number of cocycles}

Let $I_{n}$ be a matrix indexed by ${{[n]}\choose {2}}\times {{[n]}\choose{3}}$ defined as follows. Let $\sigma=\{x_0,x_1,x_2\}\subset [n]$ such that $x_0<x_1<x_2$. For a $\tau\in {{[n]}\choose {2}}$, we set
\[I_{n}(\tau,\sigma)=\begin{cases}
(-1)^i&\text{if }\tau=\sigma\setminus\{x_i\},\\
0&\text{otherwise.}
\end{cases}
\]  
Note that $I_{n}$ is just the matrix of the boundary map $\partial_2$ of the simplex on the vertex set $[n]$.

Given $X\subset {{[n]}\choose {2}}$ and $Y\subset {{[n]}\choose {3}}$, let $I_n[X,Y]$ be the submatrix of $I_n$ determined by the set rows $X$ and the set of columns $Y$.

Let $\mathcal{C}_n$ be the set of $2$-dimensional hypertrees on the vertex set $[n]$.

For a simplicial complex $K$, let $K(i)$ be the set of $i$-dimensional faces of $K$.

The next lemma is a special case of \cite[Proposition 4.2]{duval2009simplicial}.
\begin{lemma}\label{detH}
Let $F$ be a spanning tree on the vertex set $[n]$, and let $X={{[n]}\choose {2}}\setminus F$. Let $K$ be a two dimensional complex on the vertex set $[n]$ with a complete $1$-skeleton and ${{n-1}\choose 2}$ triangular faces. Then
\[\left|\det I_n[X,K(2)]\right|=\begin{cases}
|H_1(K,\mathbb{Z})|&\text{if $K\in \mathcal{C}_n$,}\\
0&\text{otherwise.}
\end{cases}\]

\end{lemma}

Let $Y\subset {{[n]}\choose {3}}$. For any $\tau\in {{[n]}\choose {2}}$, we define
\[t_Y(\tau)=|\{\sigma\in Y\,:\, \tau\subset \sigma\}|.\]

\begin{lemma}\label{upperb}
Let $Y\subset {{[n]}\choose {3}}$, then
\[\log \mathbb{P}(T_n(2)\subset Y)\le (n-2)\log (n)+\left(1-\frac{2}n\right)\sum_{\tau\in {{[n]}\choose{2}}} \log\left(\frac{t_Y(\tau)}{n}\right).\]

\end{lemma}
\begin{proof}
Let $F$ be a spanning tree on the vertex set $[n]$, and let $X={{[n]}\choose {2}}\setminus F$. Combining the Cauchy-Binet formula with Lemma~\ref{detH}, we obtain that
\[\det(I_n[X,Y](I_n[X,Y])^T)=\sum_{\substack{Z\subset Y\\|Z|={{n-1}\choose 2}}} |\det I_n[X,Z]|^2=\sum_{\substack{S\in\mathcal{C}_n\\S(2)\subset Y}} |H_1(S,\mathbb{Z})|^2.\]
Comparing this with \eqref{measuredef}, we see that
\[\mathbb{P}(T_n(2)\subset Y)=\frac{\det(I_n[X,Y](I_n[X,Y])^T)}{n^{{n-2}\choose{2}}}.\]
Note that for $\tau\in X$, we have $(I_n[X,Y](I_n[X,Y])^T)(\tau,\tau)=t_Y(\tau)$. Thus, using Hadamard's inequality, we obtain that
\begin{equation}\label{spanningtree}\mathbb{P}(T_n(2)\subset Y)\le\frac{\prod_{\tau\in X} t_Y(\tau)}{n^{{n-2}\choose{2}}}=n^{n-2}\prod_{\tau\in X}\frac{t_Y(\tau)}{n}.\end{equation}

By Cayley's formula, we have $n^{n-2}$ spanning trees on $[n]$. By double counting the pairs $(F,\tau)$, where $F$ is a spanning tree on $[n]$ and $\tau\in {{[n]}\choose {2}}\setminus F$, we see that for any given $\tau\in {{[n]}\choose{2}}$ there are $\left(1-\frac{2}n\right)n^{n-2}$  spanning trees on $[n]$ not containing $\tau$. 

Consider inequality \eqref{spanningtree} for all possible $n^{n-2}$ choices of a panning tree $F$. If we multiply these inequalities, take the logarithm and divide by $n^{n-2}$, we obtain the lemma.  
\end{proof}

Let $K$ be a simplicial complex on the vertex set $[n]$. Any cochain in $C^i(K,\mathbb{F}_2)$ is uniquely determined by its support. Thus, cochains in $C^i(K,\mathbb{F}_2)$ can be identified with subsets of $K(i)$. In particular, if $K$ has complete $1$-skeleton, then  $C^1(K,\mathbb{F}_2)$ can be identified with the set of graphs on $[n]$. Thus, the coboundary map $\delta_{K,1}$ assigns a subset of ${{[n]}\choose 3}$ to each graph on $[n]$. Let $\Delta_{[n]}$ be the simplex on $[n]$. With the conventions above, for any graph $G$ on $[n]$, we have
\begin{equation}\label{cohgraph}
    \mathbb{P}\left(G\in Z^1(T_n,\mathbb{F}_2)\right)=\mathbb{P}\left(T_n(2)\subset {{[n]}\choose 3} \setminus   \delta_{\Delta_{[n]},1} G\right).
\end{equation}

Given a graphon $W\in \mathcal{W}_0$, let $Z_W:[0,1]^2\to \mathbb{R}$ be defined as
\[Z_W=2W\circ (\mathbbm{1}-W).\]
It is easy to see that $Z_W$ is symmetric and nonnegative. Observe that
\[\mathbbm{1}-Z_W=W\circ W+(\mathbbm{1}-W)\circ (\mathbbm{1}-W).\]
Here the right hand side is clearly  nonnegative, so $Z_W\le 1$. Thus, $Z_W$ and $\mathbbm{1}-Z_W$ are both graphons. 

Let
\begin{equation}\label{fdef}f(W)=\langle W, \log\circ Z_W\rangle+\langle \mathbbm{1}-W, \log\circ(\mathbbm{1}- Z_W)\rangle.\end{equation}
Here we define $0\cdot \log(0)$ to be $0$ as it is standard practice in information theory. Note that $f(W)$ is non-positive and it might be equal to $-\infty$. 

\begin{lemma}\label{upperbf}
For any graph $G$ on the vertex set $[n]$, we have 
\[\log \mathbb{P}(G\in Z^1(T_n,\mathbb{F}_2))\le (n-2)\log (n)+\frac{n^2}2\left(1-\frac{2}n\right)f(W_G).\]
\end{lemma}
\begin{proof}
Let $Y={{[n]}\choose 3}\setminus\delta_{\Delta_{[n]},1} G$. Combining \eqref{cohgraph} with Lemma~\ref{upperb}, we see that
\begin{equation}\label{eq11}\log \mathbb{P}(G\in Z^1(T_n,\mathbb{F}_2))\le (n-2)\log (n)+\left(1-\frac{2}n\right)\sum_{\tau\in {{[n]}\choose{2}}} \log\left(\frac{t_Y(\tau)}{n}\right).\end{equation}

Let $A$ be the adjacency matrix of $G$, and let $\tau=\{u,v\}\in {{[n]}\choose{2}}$. Let $w\in [n]\setminus \tau$. 

If $\tau$ is an edge of $G$, then
\[A(u,w)(1-A(w,v))+(1-A(u,w))A(w,v)=\begin{cases} 1&\text{if }\{u,v,w\}\in Y,\\
0&\text{if }\{u,v,w\}\notin Y.
\end{cases}\]

If $\tau$ is not an edges of $G$, then 
\[A(u,w)A(w,v)+(1-A(u,w))(1-A(w,v))=\begin{cases} 1&\text{if }\{u,v,w\}\in Y,\\
0&\text{if }\{u,v,w\}\notin Y.
\end{cases}\]

Therefore, if $\tau$ is an edge of $G$, then
\begin{align*}t_Y(\tau)&=\sum_{w\in [n]\setminus \tau} \left(A(u,w)(1-A(w,v))+(1-A(u,w))A(w,v)\right)\\&\le \sum_{w\in [n]} \left(A(u,w)(1-A(w,v))+(1-A(u,w))A(w,v)\right)\\&=(2A(J-A))(u,v), 
\end{align*}
where $J$ is the matrix where each entry is $1$.

Similarly, if $\tau$ is not an edge of $G$, then
\[t_Y(\tau)\le (A^2+(J-A)^2)(u,v)=(nJ-2A(J-A))(u,v).\]

Thus, with the notation $B=2A(J-A)$, we have
\[\log\left(\frac{t_Y(\tau)}{n}\right)\le A(u,v)\log\left(\frac{B(u,v)}{n}\right)+(1-A(u,v))\log\left(1-\frac{B(u,v)}{n}\right).\]

It is easy to check that for any $u\in [n]$, we have $1-\frac{B(u,u)}{n}=1$. Thus,
\[A(u,u)\log\left(\frac{B(u,u)}{n}\right)+(1-A(u,u))\log\left(1-\frac{B(u,u)}{n}\right)=0.\]

Therefore,
\begin{align}\label{eq12}
\sum_{\tau\in{{[n]}\choose{2}}} \log\left(\frac{t_Y(\tau)}{n}\right)&=\frac{1}2\sum_{\substack{u,v\in [n]\\u\neq v}}\log\left(\frac{t_Y(\{u,v\})}{n}\right)\\&\le \frac{1}2\sum_{\substack{u,v\in [n]\\u\neq v}} \left(A(u,v)\log\left(\frac{B(u,v)}{n}\right)+(1-A(u,v))\log\left(1-\frac{B(u,v)}{n}\right)\right)\nonumber\\&=\frac{1}2\sum_{\substack{u,v\in [n]}}\left( A(u,v)\log\left(\frac{B(u,v)}{n}\right)+(1-A(u,v))\log\left(1-\frac{B(u,v)}{n}\right)\right)\nonumber\\&=\frac{n^2}2 f(W_G),\nonumber
\end{align}
where the last equality is just straightforward from the definitions.

The lemma follows by combining \eqref{eq11} and \eqref{eq12}. 
\end{proof}
\section{The proof of Theorem~\ref{thm1}}
Recall that we defined $f$ in \eqref{fdef}.

\begin{lemma}\label{semic0}
The function $f$ is upper semicontinuous on $\mathcal{W}_0$ with respect to the cut norm.
\end{lemma}
\begin{proof}
For $k>0$, we define $\ell_k:[0,1]\to \mathbb{R}$ as $\ell_k(x)=\max(-k,\log(x))$. Given a graphon $W\in \mathcal{W}_0$, let
\[f_k(W)=\langle W, \ell_k\circ Z_W\rangle+\langle \mathbbm{1}-W, \ell_k\circ(\mathbbm{1}- Z_W)\rangle.\]

By the monotone convergence theorem, we have $f(W)=\inf_{k>0} f_k(W)$. So it is enough to prove that $f_k$ is continuous for all $k>0$. 

Let $W_1,W_2,\dots$ be a sequence of graphons converging to the graphon $W$ in the cut norm. Using Lemma~\ref{limitofproduct}, we see that $\lim_{n\to\infty} Z_{W_n}=Z_W$ in the $L_1$-norm. Since $\ell_k$ is a Lipschitz function, it follows from Lemma~\ref{Lipschitz} that  $\lim_{n\to\infty}\ell_k\circ Z_{W_n}=\ell_k\circ Z_W$  in the $L_1$-norm. 

Finally, we can apply Lemma~\ref{lemmaangle} to conclude that $\lim_{n\to\infty}\langle W_n, \ell_k\circ Z_{W_n}\rangle=\langle W, \ell_k\circ Z_W\rangle$. A similar argument gives that $\lim_{n\to\infty}\langle \mathbbm{1}-W_n, \ell_k\circ(\mathbbm{1}- Z_{W_n})\rangle=\langle \mathbbm{1}-W, \ell_k\circ(\mathbbm{1}- Z_W)\rangle$. Thus, $\lim_{n\to\infty} f_k(W_n)=f_k(W)$. Therefore, $f_k$ is indeed continuous.
\end{proof}

\begin{remark} The function $f$ is not continuous on $\mathcal{W}_0$ with respect to the cut norm as the following example shows. Let $W_n$ be the indicator function of $[0,n^{-1}]^2$, then $W_n$ converge to the zero graphon $0\cdot \mathbbm{1}$. We have $f(W_n)=-\infty$ for all $n$, but $f(0\cdot \mathbbm{1})=0$. 
\end{remark}
\begin{lemma}\label{semic}
The function $f$ is a well-defined upper semicontinuous function on $\widetilde{\mathcal{W}}_0$.
\end{lemma}
\begin{proof}
To prove that $f$ is well defined we need to show that if $V$ and $W$ are two graphons such that $V\in \widetilde{W}$, then $f(V)=f(W)$. Since $V\in \widetilde{W}$, we can choose a sequence $\varphi_n\in S_{[0,1]}$ such that $W^{\varphi_n}$ converge to $V$ in the cut norm. It is clear from the definition that $f(W^{\varphi_n})=f(W)$. Combining this with Lemma~\ref{semic0}, it follows that $f(V)\ge f(W)$. Since $W\in \widetilde{V}$ also holds, the same argument gives that $f(W)\ge f(V)$. Thus, $f(V)=f(W)$. So $f$ is indeed well-defined. The semicontinuity is clear from Lemma~\ref{semic0}.  
\end{proof}

For $u\ge 0$, let
\[\widetilde{F}_u=\left\{\widetilde{W}\in \widetilde{\mathcal{W}}_0\,:\,f(\widetilde{W})\ge -u\right\}.\]

\begin{lemma}\label{LDFu}
For $u\ge 0$, we have
\[\limsup_{n\to\infty} \frac{1}{n^2}\log |\mathcal{G}(n,\widetilde{F}_u)|\le \frac{u}2.\]

\end{lemma}
\begin{proof}
    Using Lemma~\ref{semic}, we see that $\widetilde{F}_u$ is closed. Thus, by Lemma~\ref{logcount}, we need to show that $H(W)\le u$ for all $W\in \mathcal{W}_0$ such that $f(W)\ge -u$. This statement will follow once we prove that $f(W)+H(W)\le 0$ for all $W\in \mathcal{W}_0$. This inequality holds, because
    \begin{align*}
    f(W)+H(W)&=\int_{[0,1]^2} W(x,y)\log\left(\frac{Z_W(x,y)}{W(x,y)}\right)+(1-W(x,y))\log\left(\frac{1-Z_W(x,y)}{1-W(x,y)}\right) dx dxy\\&=\int_{[0,1]^2}-\dkl\left(\{W(x,y),1-W(x,y)\}\, \Big|\Big|\, \{Z_W(x,y),1-Z_W(x,y)\}\right)dxdy\\&\le 0,
    \end{align*}
    where $\dkl$ denotes the Kullback–Leibler divergence and the last inequality follows from Gibbs' inequality.
\end{proof}

Note that $H_1(T_n,\mathbb{F}_2)\cong H^1(T_n,\mathbb{F}_2)$, see \cite[Chapter 3.1]{hatcher2000algebraic}. It will be more convenient to work with $H^1(T_n,\mathbb{F}_2)$. We have
\begin{equation}\label{Eh1}\mathbb{E} |H_1(T_n,\mathbb{F}_2)|=\mathbb{E} |H^1(T_n,\mathbb{F}_2)|\le \mathbb{E} |Z^1(T_n,\mathbb{F}_2)|=\sum_G \mathbb{P}\left(G\in Z^1(T_n,\mathbb{F}_2)\right),\end{equation}
where the summation is over all the graphs on $[n]$.

Let $k$ be a large positive integer, and let $\varepsilon=\frac{\log 2}k$. For $i=0,1,2,\dots,k-1$, let
\[\widetilde{L}_i=\left\{\widetilde{W}\in \widetilde{\mathcal{W}}_0\,:\: -i\varepsilon \ge f(\widetilde{W})>-(i+1)\varepsilon \right\},\]
and let $\widetilde{L}_{k}=\left\{\widetilde{W}\in \widetilde{\mathcal{W}}_0\,:\: -k\varepsilon \ge f(W) \right\}$. 

Since $f$ is non-positive, for every $n$, $(\mathcal{G}(n,\widetilde{L}_i))_{i=0}^{k}$ gives us a partition of the set of all the graphs on $[n]$. Using Lemma~\ref{upperbf}, we see that if $n$ is large enough, then for all $G\in \mathcal{G}(n,L_i)$, we have
\begin{align}\label{Pbound}\mathbb{P}\left(G\in Z^1(T_n,\mathbb{F}_2)\right)&\le  (n-2)\log (n)+\frac{n^2}2\left(1-\frac{2}n\right)f(W_G)\\&\le(n-2)\log (n)-\frac{n^2}2\left(1-\frac{2}n\right)i\varepsilon\nonumber\\&\le -\frac{n^2(i-1)\varepsilon}2. \nonumber
\end{align}

Let $0\le i\le k-1$. We have $\widetilde{L}_i\subset \widetilde{F}_{(i+1)\varepsilon}$, so Lemma~\ref{LDFu} gives us that for all large enough $n$, we have
\begin{equation}\label{GLi}\log |\mathcal{G}(n,\widetilde{L}_i)|\le \log |\mathcal{G}(n,\widetilde{F}_{(i+1)\varepsilon})|\le \frac{n^2(i+2)\varepsilon}2.
\end{equation}

Since we have less than $2^{\frac{n^2}2}$ graphs on $[n]$, \eqref{GLi} also holds for $i=k$. 

Thus, combining \eqref{Eh1}, \eqref{Pbound} and \eqref{GLi}, we see that for all large enough $n$, we have
\begin{align*}\mathbb{E} |H_1(T_n,\mathbb{F}_2)|&\le\sum_{i=0}^k \sum_{G\in \mathcal{G}(n,\widetilde{L}_i)}\mathbb{P}\left(G\in Z^1(T_n,\mathbb{F}_2)\right)\\&\le \sum_{i=0}^k \exp\left(\frac{n^2(i+2)\varepsilon}2\right)\exp\left(-\frac{n^2(i-1)\varepsilon}2\right)\\&\le (k+1)\exp\left(2\varepsilon n^2 \right).
\end{align*}

Combining this with Markov's inequality, for all large enough $n$, we have
\begin{align*}\mathbb{P}(\dim H_1(T_n,\mathbb{F}_2)\ge 4\varepsilon n^2)&=\mathbb{P}(| H_1(T_n,\mathbb{F}_2)|\ge \exp(4\log(2)\varepsilon n^2))\\&\le \frac{\mathbb{E} |H_1(T_n,\mathbb{F}_2)|}{\exp(4\log(2)\varepsilon n^2)}\\&\le (k+1)\exp((2-4\log(2))n^2), \end{align*}
which tends to $0$. Since we can choose $k$ to be arbitrary large Theorem~\ref{thm1} follows.

\section{The proof of Theorem~\ref{thm2}}

The proof of Theorem~\ref{thm2} is almost identical to the proof of Theorem~\ref{thm1}, so we only sketch it.

It is easy to see and it was also observed by Linial and Peled \cite[Section 5]{linial2019enumeration} that if $G$ is a graph on the vertex set $[n]$, then
\[\mathbb{P}\left(G\in Z^1(S_2(n,1),\mathbb{F}_2)\right)=\prod_{\tau\in{{[n]}\choose{2}}} \frac{t_Y(\tau)}{n-2},\]
where $Y={{[n]}\choose 3}\setminus\delta_{\Delta_{[n]},1} G$. 

Thus,
\begin{align*}\log \mathbb{P}\left(G\in Z^1(S_2(n,1),\mathbb{F}_2)\right)&={{n}\choose{2}}\log\left(1+\frac{2}{n-2}\right)+\sum_{\tau\in{{[n]}\choose{2}}} \log \frac{t_Y(\tau)}{n}\\&\le \frac{n(n-1)}{n-2}+\frac{n^2}2f(W_G),\end{align*}
where the last inequality follows from \eqref{eq12}.

The proof can be finished along the lines of the proof of Theorem~\ref{thm1}.

\bibliography{references}
\bibliographystyle{plain}



\end{document}